\documentclass[leqno,11pt,english]{amsart}
\usepackage{amsthm}
\usepackage{amstext}
\usepackage{amssymb}
\usepackage{mathrsfs}
\usepackage{amsmath}
 
\usepackage{xcolor}
\usepackage[all]{xy}
\usepackage{babel}
 \newtheorem{theorem}{Theorem}[section]

\newtheorem{proposition}[theorem]{Proposition}

\newtheorem{lemma}[theorem]{Lemma}
\theoremstyle{definition}
\newtheorem{definition}[theorem]{Definition}
\newtheorem{remark}[theorem]{Remark}
\numberwithin{equation}{section}
\newtheorem{example}[theorem]{Example}
\def \C {\mathbb{C}}

\def \fol {{\mathscr{F}}}

\def \P {\mathbb{P}}

\def \sing {{\rm Sing}}
\def \dim {{\rm dim}}

\begin{document}

\title[On the Milnor number of one-dimensional foliations]{On the Milnor number of non-isolated singularities of holomorphic foliations and its topological invariance}

\author[A. Fern\'andez-P\'erez]{Arturo Fern\'andez-P\'erez}
\address{Arturo Fern\'andez P\'erez \\
ICEx - UFMG \\
Departamento de Matem\'atica \\
Av. Ant\^onio Carlos 6627 \\
30123-970 Belo Horizonte MG, Brazil} 
\email{fernandez@ufmg.br}

\author{Gilcione Nonato Costa}
\address{Gilcione Nonato Costa \\
ICEx - UFMG \\
Departamento de Matem\'atica \\
Av. Ant\^onio Carlos 6627 \\
30123-970 Belo Horizonte MG, Brazil} \email{gilcione@mat.ufmg.br}
\author{Rudy Rosas Baz\'an}
\address{Rudy Rosas Baz\'an\\ Dpto. Ciencias - Secci\'on Matem\'aticas, 
Pontif\'icia Universidad Cat\'olica del Per\'u, Av Universitaria 1801, Lima, Per\'u.}
\email{rudy.rosas@pucp.pe}

\thanks{
The first-named author is supported by CNPq-Brazil Grant Number 302790/2019-5 and Pronex-Faperj. The third-named is supported by Vicerrectorado de Investigaci\'on de la Pontificia Universidad Cat\'olica del Per\'u.}


\begin{abstract}
We define the Milnor number -- as the intersection number of two holomorphic sections --  of a one-dimensional holomorphic foliation $\fol$ with respect to a compact connected component $C$ of its singular set. Under certain conditions, we prove that the Milnor number of $\fol$ on a three-dimensional manifold with respect to $C$ is invariant by $C^1$ topological equivalences. 
%As consequence, we obtain the same conclusion for Milnor-Parusi\'nski-Aluffi's number of complex hypersurfaces.
\end{abstract}
\maketitle

\section{Introduction}
\par One of the most studied invariants in Singularity theory is the \textit{Milnor number} of a complex hypersurface, such number was defined by Milnor \cite{milnor}. In Foliation theory, for a holomorphic vector field $v=P(x,y)\frac{\partial}{\partial{x}}+Q(x,y)\frac{\partial}{\partial{y}}$ in $\C^2$, the Milnor number arises initially as the \textit{intersection number} of the curves $P(x,y)=0$ and $Q(x,y)=0$ in Seidenberg \cite{seidenberg} and Van den Essen \cite{essen}, however, both authors did not call it Milnor number. The first authors to establish the denomination of \textit{Milnor number for holomorphic foliations} have been 
Camacho, Lins Neto, and Sad in \cite{CSL}. They proved that the Milnor number of a one-dimensional holomorphic foliation with an isolated singularity is a \textit{topological invariant}, see \cite[Theorem A, p. 149]{CSL}. 
%foliation induced by a \textit{gradient vector field} of a holomorphic function, this number coincides with the usual Milnor number for complex hypersurfaces. 
\par The proposal to investigate the Milnor number of foliations with non-isolated singularities arises naturally. In the case of complex hypersurfaces, such a study was done by Parusi\'nski \cite{parusinski} and later generalized to the category of schemes by Aluffi \cite{aluffi}. Motivated by these studies, in this paper we will adapt Parusi\'nski's definition to define the Milnor number of a one-dimensional holomorphic foliation with non-isolated singularities. For foliations on smooth algebraic varieties, we will apply \textit{Fulton's intersection theory} \cite{fulton} to obtain an explicit formula for the Milnor number of a foliation in terms of the Chern and Segre classes. In a similar way to Parusi\'nski  and Aluffi, our definition is given by the intersection number of the two sections of a holomorphic vector bundle associated to foliation, specifically such an intersection number is known in Fulton's theory by the \textit{excess intersection}.
\par Before establishing the objectives of this paper, we will give some notations and results. Let $\fol$ be a one-dimensional holomorphic foliation in an open subset $U$ of $\C^n$ induced by a holomorphic vector field $v$ in $U$. 
The \textit{Milnor number} of $\fol$ at $p\in U$ is 
\begin{equation}\label{milnor}
\mu(\fol,p)=\dim_{\mathbb{C}}\frac{\mathcal{O}_{n,p}}{\displaystyle(P_1,\ldots,P_n)}
\end{equation}
where $\mathcal{O}_{n,p}$ is the ring of germs of holomorphic functions at $p$ and $(P_1,\ldots,P_n)$
is the ideal generated by the germs at $p\in U$ of the coordinate functions of $v$. Note that $\mu(\fol,p)$ is finite if and only if $p\in U$ is an isolated singularity of $v$. Moreover, it follows from \cite[p. 123]{fulton} that the Milnor number of $\fol$ at $p$ agrees with the \textit{intersection number} at $p$ of the divisors $D_i=\{P_i=0\}$ $\forall$ $i=1,\ldots,n$, i.e.,
\[\mu(\fol,p)=i_p(D_1,\ldots,D_n).\]
\par When $p$ is an isolated singularity of $\fol$, 
$\mu(\fol,p)$ is a \textit{topological invariant} of $\fol$ provided that $n\geq 2$ as proved in \cite[Theorem A]{CSL}. More specifically, if $\fol$ and $\fol'$ are one-dimensional holomorphic foliations locally topologically equivalent at $p$ and $p'$ respectively, that is, there is a homeomorphism $\phi$ between neighborhoods of $p$ and $p'$ taking leaves of $\fol$ to leaves of $\fol'$ with $\phi(p)=p'$. Then \[\mu(\fol,p)=\mu(\fol',p').\] 

\par Our first aim is to define the \textit{Milnor number of foliations with non-isolated singularities}. We will give a definition that works for any compact connected component of the singular set of such a foliation (see Definition \ref{milnor_defi}), 
and we will show that it is a generalization of the usual Milnor number of a foliation with an isolated singularity. 
\par Second, motivated by \cite{CSL}, 
we study the \textit{topological invariance problem} for the Milnor number of a one-dimensional foliation with a non-isolated singularity. Under some conditions, we solve the problem for one-dimensional foliations on three-dimensional complex manifolds (see Theorem \ref{c1inv}), and moreover we explain the reason why our proof does not adapt to arbitrary dimensions (see Remark \ref{obs1}). 
\par We remark that if $X$ is a complex hypersurface with non-isolated singularities on a complex manifold $M$, the difference $c_i^{SM}(X)-c_i^{FJ}(X)$ of the Schwartz-MacPherson and Fulton-Johnson classes is, for each $i$, a homology class with support in the homology $H_{2i}(\sing(X))$ of the singular set of $X$. This difference is called the \textit{Milnor class of degree $i$} of $X$, see for instance \cite{Brasselet} and \cite[p. 194]{libro_Bra}. The $0$-degree class of $c_i^{SM}(X)-c_i^{FJ}(X)$ coincides with the Milnor-Parusi\'nski-Aluffi's number of a complex hypersurface defined in \cite{aluffi} and \cite{parusinski}. In a future paper, we hope to study and describe the Milnor class of a one-dimensional holomorphic foliation on a complex manifold.  
\par The paper is organized as follows: in Section \ref{defi}, we define the concept of one-dimensional holomorphic foliations on complex manifolds. Section \ref{def_milnor} is devoted to the definition of the Milnor number (as intersection number) of a one-dimensional foliation with respect to a compact connected component of its singular set. In Section \ref{exa}, we give two examples of foliations on the three-dimensional complex projective space with non-isolated singularities where its Milnor number is exhibited. 
In Section \ref{The Milnor}, we explain the reason why the Milnor number of foliations along non-isolated singularities can not be defined as the Poincar\'e-Hopf index of any vector field. In Section \ref{Theorem_1}, we prove the main result of the paper  Theorem \ref{c1inv}. Such theorem asserts, under certain conditions, that the Milnor number of a foliation is a topological invariant. 
%Finally, in Section \ref{paru}, we obtain that the Milnor number of a complex hypersurface defined by Parusi\'nski and Aluffi \cite{aluffi,parusinski} is a topological invariant in the case of three-dimensional complex manifolds. 

\section{One-dimensional holomorphic foliations}\label{defi}
Let $M$ be an $n$-dimensional complex manifold. A one-dimensional holomorphic foliation $\fol$ on $M$ may be defined as follows: we take an open covering $\{U_j\}_{j\in I}$ of $M$ and on each $U_j$ a holomorphic vector field $v_j$ with zeros set of codimension at least 2, and we require that on $U_j\cap U_i$ the vector fields $v_j$ and $v_i$ coincide up to multiplication by a nowhere vanishing holomorphic function:
\[v_i=g_{ij}v_j\,\,\,\,\,\,\,\text{on}\,\,U_i\cap U_j,\,\,\,\,\,\,g_{ij}\in\mathcal{O}^{*}_M(U_i\cap U_j).\]
This means that the local integral curves of $v_i$ and $v_j$ glue together, up reparametrization, giving the so-called \textit{leaves} of $\fol$. Then $\fol$ is an equivalence class of collection $\{U_j,v_j\}_{j\in I}$, where the equivalence relation is given by:  $\{U_j,v_j\}_{j\in I}\sim\{U'_j,v'_j\}_{j\in I'}$ if $v_j$ and $v'_i$ coincide on $U_j\cap U'_{i}$ up to multiplication by a nowhere vanishing holomorphic function. The \textit{singular set} $\sing(\fol)$ of $\fol$ is the complex subvariety of $M$ defined by 
\[\sing(\fol)\cap U_j:=\sing(v_j),\,\,\,\,\,\,\forall j\in I.\]
\par The functions $g_{ij}\in\mathcal{O}^*_M(U_i\cap U_j)$ form a multiplicative cocycle and hence give a cohomology class in $H^{1}(M,\mathcal{O}^*_M)$, that is a line bundle on $M$ so called \textit{cotangent bundle} of $\fol$, and denoted by $T^*_{\fol}$. Its dual $T_\fol$, is represented by the inverse cocycle $\{g_{ij}^{-1}\}$, is called \textit{tangent bundle} of $\fol$.
\par The relations $v_i=g_{ij}v_j$ on $U_i\cap U_j$ can be glued to a global holomorphic section $s$ of $TM\otimes T^*_{\fol}$. Since each $v_j$ has zeros set of codimension at least 2, $s$ also has zeros set of codimension at least 2. Note that $s$ is not entirely intrinsically defined  by $\fol$: if we change from $\{U_j,v_j\}_{j\in I}$ to $\{U_j,fv_j\}_{j\in I}$, where $f\in\mathcal{O}^{*}_M(M)$, then $s$ will be replaced by $fs$. But it is not problem to definition of $\fol$ by global sections and only ambiguity. 
\par A complex hypersurface $V$ in $M$ is said to be \textit{invariant by $\fol$} if 
\[v_j(f_j)=h_j f_j\,\,\,\,\,\,\,\,\,\forall j\in I,\]
where $V\cap U_j=\{f_j=0\}$ and $h_j\in\mathcal{O}(U_j)$.
\section{The Milnor number as intersection number}\label{def_milnor}

Let $\fol$ be a one-dimensional holomorphic foliation on $M$. Suppose that $\textrm{Sing}(\fol)$  has complex codimension at least $2$. By definition, $\fol$ is given by a section $s:M\to E:=TM\otimes T^*_{\fol}$ with zero set $\textrm{Sing}(\fol)$. Let $C$ be a compact connected component of $\sing(\fol)$, we have the fiber square 
\[\xymatrix{
{C}\ar[d]_i\ar[r]^{i} &
M\ar[d]^{s} \\
M\ar[r]_{s_0} & E
}\]
where $i$ is the canonical inclusion and $s_0$ is the zero section of $E$. 
Let $U$ be a small neighborhood of $C$. We follow Parusi\'nski \cite[p. 248]{parusinski} to consider $ind_U(s)$ the \textit{intersection number} over $U$ of $s$ and the zero section $s_0$ of $E$. 
Parusi\'nski remarked that if $s'$ is a small perturbation of $s$ transversal to the zero section, then $ind_U(s)$ equals the number of zeros of $s'$ counted with signs (local indices). 
Moreover, $ind_U(s)$ depends only on the homotopy class of $s|_{\partial{U}}$ in the space of nowhere zero section of $E|_{\partial{U}}$ and if $E$ is trivial this definition agrees with that of the topological degree (see for instance \cite{amann}). Using standard homotopy arguments, it is easy to prove that $ind_U(s)$ depends only on $s$ and $C$, so in order to standardize the notation with the intersection theory in Algebraic Geometry, we shall denote $i_C(s,s_0)$ the number $ind_U(s)$. 
\begin{definition}\label{milnor_defi}
We define the \textit{Milnor number of $\fol$ at $C$} by
\[\mu(\fol,C)=i_C(s,s_0).\]
\end{definition}
\begin{remark}
If $M$ is a smooth $n$-dimensional algebraic variety and $C$ is smooth, 
we can apply Fulton's intersection theory (see for instance \cite[Proposition 6.1]{fulton} or \cite[p. 328]{aluffi}) to define the \textit{intersection number} between the sections $s_0$ and $s$ along $C$ as
\begin{equation}\label{ful_1}
i_C(s,s_0)=\{c(TM\otimes T^*_{\fol})\cap s(C,M)\}_0\in A_0(C),
\end{equation}
where $A_0(C)$ is the Chow group of $S$ of degree zero. Here and in the following $c$ denotes \textit{total Chern class} and $s$ is the \textit{Segre class} (in the sense of \cite{fulton}), moreover pullback notations are omitted when there is no ambiguity. According to Fulton \cite[p. 153]{fulton} (see also Eisenbud-Harris \cite[p. 458]{eisenbud}), it definition of Minor number for $\fol$ works for compact connected components of $\sing(\fol)$.
\end{remark}
%\begin{definition}
%We define the \textit{Milnor number of $\fol$ at $S$} as 
%\[\mu(\fol,S):=i_S(s,s_0).\]
%\end{definition}
\begin{remark}
 If  $C=\{p\}$ is an isolated singularity of $\fol$, it follows from \cite[Proposition 8.2]{fulton} that 
 \[\mu(\fol,p)=length(\mathcal{O}_{s\cap s_0,p}).\]
If $s$ is locally generated at $p$ by 
$v=P_1(z_1,\ldots,z_n)\frac{\partial}{\partial{z_1}}+\ldots+P_n(z_1,\ldots,z_n)\frac{\partial}{\partial{z_n}},$
where $P_1\ldots,P_n\in\mathcal{O}_{n,p}$, we get
\[\mu(\fol,p)=\dim_{\C}\frac{\mathcal{O}_{n,p}}{(P_1,\ldots,P_n)}.\]
\end{remark}
\par If the singularities of $\fol$ are all isolated, the Baum-Bott formula \cite{BB} says
\begin{equation}\label{Baum}
\sum_{p\in\sing(\fol)}\mu(\fol,p)=c_n(TM\otimes T^*_{\fol})\cap [M],
\end{equation}
where $c_n$ denotes the top Chern class. 
\par Let us consider $M=\P^n$ and $\fol$ be a one-dimensional foliation of degree $d$ in $\P^n$. The degree $d$ of $\fol$ is the number of tangencies between $\fol$ and a generic hyperplane. It is not difficult to prove that $T^{*}_{\fol}=\mathcal{O}(d-1)$ so that $\fol$ is given by a global section $s$ of $T\P^n(d-1)$. When all the singularities of $\fol$ are isolated, we get from equality \ref{Baum} that
\[\sum_{p\in\sing(\fol)}\mu(\fol,p)=\sum^{n}_{i=0}d^i.\]
\par On the other hand, when the scheme singular of $\fol$ is formed by a disjoint union of proper smooth subschemes $C$ and $F$, where $F$ finite, then it follows from Vainsencher \cite[p. 81]{israel} that
\[\mu(\fol,C)=\sum_{i\geq 0}\int_{\P^n}c_{n-(1+i)}(T\P^n(d-1)) s_{1+i}(C,\P^n)\]
and 
\[\sum_{p\in F}\mu(\fol,p)+\mu(\fol,C)=\sum^{n}_{i=0}d^{i}.\]
\par For a study on the number of residual isolated singularities of foliations on complex projective spaces, see \cite{costa, gilcione, arturo} and the interested reader may consult \cite{cavalcante, jardim} for results about the classification of one-dimensional foliations of low degree on threefolds. 
\section{Examples}\label{exa}
%In order to compute explicitly the Minor number of foliations along nonisolated singularities we will use the \textit{principle of continuity} and \textit{Fulton's dynamic intersections} (see for instance \cite[\S10 and \S 11]{fulton} and \cite[p. 456]{eisenbud}). 
%\par Let $\fol$ be a one-dimensional foliation on $M$ and $s$ be the section of $TM\otimes T^{*}_{\fol}$ defining $\fol$. Then the principle of continuity asserts that we 
%may deform $s$ to $s'_t$ with $s'_0=s$, where $\{s'_t\}_{t\in T}$ is a family of sections of $TM\otimes T^{*}_{\fol}$ parametrized by a non-singular curve $T$ with $0\in T$, such that $s'_t$ meets $s_0$ properly for generic $t$ at points $p_1^{t},\ldots p_k^{t}$ and as consequence we obtain 
%\[\mu(\fol,S)=\lim_{t\to 0}\sum_{p^{t}_j\to S}i_{p^{t}_j}(s'_t,s_0).\]
%\par We refer \cite{bracci} to see a recent theory of deformation (or principle of continuity) of foliations. 
In this section, we give some examples of one-dimensional holomorphic foliations on $\P^3$ with non-isolated singularities where its Milnor number is computed.
\begin{example} Let us consider the
foliation $\fol_0$ of degree $2$ in $\P^3$ defined in the
open affine set $U_3=\{[\xi_0:\xi_1:\xi_2:\xi_3]\in\P^3:\xi_3\neq0\}$ by the vector
field
$$X_0(z)=z_1^2\frac{\partial}{\partial z_1}+z_1^2\frac{\partial}{\partial
z_2}+z_2^2\frac{\partial}{\partial z_3}$$ where
$z_i=\xi_{i-1}/\xi_3$ for $i=1,2,3$.
Let $C=\{\xi_0=\xi_1=0\}$, then the singular set of $\fol_0$
is 
\[\sing(\fol_0)=C\cup\{p\},\] where $p=[1:1:1:0]$. It is not difficult  to see $\mu(\fol_0,p)=1$ which implies that $$\mu(\fol_0,C)=14.$$ In fact, let $\fol_t$ be a generic
perturbation of $\fol_0$, $0 < |t| < \epsilon$, with $\epsilon$ sufficiently small, described in $U_3$ by the vector field $X_t$ as follows
$$ X_t=X_0(z)+t\sum_{i=1}^{3}\sum_{j=0}^{2}P_{ij}(z_1,z_2,z_3)\frac{\partial}{\partial z_i}$$
where $P_{ij}$ are homogeneous polynomial of degree $j$ for all $i=1,2,3$.
Note that $\fol_t$ has degree 2, for all $t\ne0$. For $P_{ij}$ generic polynomials, Bezout theorem implies that $\fol_t|_{U_3}$ contains 8 isolated points, counted with multiplicities. Let $z_k^t= (z_{1k}^t,z_{2k}^t, z_{3k}^t)$ be one these points, where $k=1,\ldots,8$. Therefore, we have

$$\lim_{t\to0}\bigg( (z_{1k}^t)^2+t\sum_{j=0}^{2}P_{1j}(z_k^t)\bigg)=\lim_{t\to0}  (z_{1k}^t)^2=0 $$
which implies that 
$$ \lim_{t\to0}  z_{1k}^t=0.$$
In the same way, we can conclude $$\lim_{t\to0}  z_{2k}^t=0,$$ i.e.,
$$\lim_{t\to0} z_{k}^t \in C,\,\,\,\,\,\,\,\,\,\, \forall\,\,\, k=1,\ldots,8.$$
The infinite hyperplane $H_3=\P^3\setminus U_3$ is an invariant hypersurface by $\fol_t$ which on it  is described by the vector field
$$Y_t = \bigg(u_1^2-u_1u_2^2+tQ_1(u)\bigg)\frac{\partial}{\partial u_1}+ \bigg(u_1^2-u_2^3+tQ_2(u)\bigg)\frac{\partial}{\partial u_2}$$
where $Q_i(u)=P_{i2}(u_1,u_2,1)-u_iP_{32}(u_1,u_2,1)$, $u_i=\xi_{i-1}/\xi_2$ for $i=1,2$.
On $H_3$, there are 7 singular points of $\fol_t$, counted the multiplicities.  Let $u_k^t= (u_{1k}^t,u_{2k}^t)$ be one these points. In order to compute these singular points, we must solve the following system
$$\left\{\begin{array}{l}
             (u_{1k}^t)^2-(u_{1k}^t)(u_{2k}^t)^2+tQ_1(u_k^t)=0\cr
             (u_{1k}^t)^2-(u_{2k}^t)^3+tQ_2(u_k^t)=0.
            \end{array}\right.$$
With these two equations, we get 
$$ (u_{2k}^t)^3-u_{1k}^t(u_{2k}^t)^2+t(Q_1(u_k^t)-Q_2(u_k^t))=0$$
where we obtain a expression for $u_{1k}^t$. By replacing this expression of $u_{1k}^t$ in the second equation of the above system, we get the following equation
$$ (u_{2k}^t)^6- (u_{2k}^t)^7+tQ_2(u_k^t) (u_{2k}^t)^4+2t (u_{2k}^t)^3(Q_1(u_k^t)-Q_2(u_k^t))+t^2(Q_1(u_k^t)-Q_2(u_k^t))^2=0.$$
            Let 
$(u_{1k},u_{2k})=\displaystyle\lim_{t\to0}u_k^t= \lim_{t\to0}(u_{1k}^t,u_{2k}^t).$
We get 
$$(u_{2k})^6- (u_{2k})^7=0,$$
which implies that either $u_{2k}=0$ or $u_{2k}=1$. If $u_{2k}=0$ then $u_{1k}=0$ and if $u_{2k}=1$ then $u_{1k}=1$ since $u_{1k}^2-u_{1k}u_{2k}^2=u_{1k}^2-u_{2k}^3=0$. 
Hence, we have two possibilities either $\displaystyle\lim_{t\to0}u_k^t=(0,0)$ or $\displaystyle\lim_{t\to0}u_k^t=(1,1)$. The point $p=[1:1:1:0]$ corresponds to $(1,1)$ and $q=[0:0:1:0]\in C$ corresponds to $(0,0)$. Finally, it is not difficult to see that $\mu(\fol_t|_{H_3},q)=6$ which results
$$\mu(\fol_0,C)=(8+6)=14.$$
\end{example}

\begin{example}
Let $\fol$ be the holomorphic foliation defined in $\P^3$  defined in the affine open set $U_3=\{[\xi_0:\xi_1:\xi_2:\xi_3]\in\P^3:\xi_3\ne 0\}$ by the vector field
\begin{eqnarray*}
X_0&=&\bigg(a_0z_1(z_3-1)+a_1z_2(z_1-1)\bigg)\frac{\partial}{\partial z_1}+\bigg(b_0z_1(z_1-1)+b_1z_2(z_3-1)\bigg)\frac{\partial}{\partial
z_2}\\
& &+ z_1\bigg(c_0(z_1-1)+c_1(z_3-1)\bigg)\frac{\partial}{\partial z_3}
\end{eqnarray*}
 where
$z_i=\xi_{i-1}/\xi_3$ for $i=1,2,3$ and $a_i,b_i,c_i$ are non-null complex numbers such that the singular set of $\fol$ consists of two curves $C_1$ and $C_2$ defined by
$$ C_1=\{\xi_0=\xi_1=0\},\,\,\,\,\,\,\,C_2=\{\xi_0-\xi_3=\xi_2-\xi_3=0\}$$ and four isolated points on the hyperplane $H_3=\P^3\setminus U_3$. 
In order to compute the Milnor numbers of $C_1$ and $C_2$, we will use the perturbation $\fol_t$ of $\fol$ which is described in $U_3$ by the vector field $X_t$ given by
$$X_t=X_0 +t\bigg( A(z)\frac{\partial}{\partial z_2}+B(z)\frac{\partial}{\partial z_3}\bigg)$$ 
where $A(z)=\alpha_0z_1^2+\alpha_1z_1z_2+\alpha_2 z_2^2$ and $B(z)=\beta_0(z_1-1)^2+\beta_1(z_1-1)(z_3-1)+\beta_2(z_3-1)^2)$ are generic quadratic functions. 
Note that, $C_1$ and $C_2$ are invariant curves of $\fol_t$ for $t\ne 0$. By Baum-Bott's formula,  there are 3 isolated points of $\sing(\fol_t)$ on $C_1$ and $C_2$ for $t\ne 0$. In fact, the points $p_1=[0:0:1:0]\in H_3$, $p_2=[0:0:z_{31}:1]$ and $p_3=[0:0:z_{32}:1]$ belong to $C_1$, with $B(0,0,z_{3i})=0$, $i=1,2$; and the points $p_4=[0:1:0:0]\in H_3$, $p_5=[1:z_{21}:1:1]$ and $p_6=[1:z_{22}:1:1]$ belong to $C_2$, with $A(1,z_{2i},1)=0$, $i=1,2$. Therefore, the  singular set of $\fol_t$  contains 8 isolated points in the affine open set $U_3$, counting the multiplicities. More precisely, two of these 8 points are on $C_1$, namely, $p_2$ and $p_3$; two these points are on $C_2$, namely $p_5$ and $p_6$. Furthermore, that two of these 8 points converge to $C_1$  and two these points converge to $C_2$ when $t$ tends to $0$. In fact, let $z_t=(z_{1t},z_{2t},z_{3t})\in \big(\sing(\fol_t)\setminus\{C_1\cup C_2)\}\big)\cap U_3$. Thus, we can write $z_{1t}=\lambda_t z_{2t}$ and $z_{3t}-1=\eta_t(z_{1t}-1)$ which results $a_0\lambda_t\eta_t+a_1=0$ and
\begin{equation}\label{sist1}
\left\{\begin{array}{l}
               (b_0\lambda_t+b_1\eta_t)(z_{1t}-1)+ta(\lambda_t)z_{2t}=0\cr
               \lambda_t(c_0+c_1\eta_t)z_{2t}+tb(\eta_t)(z_{1t}-1)=0
               \end{array}\right.
\end{equation}
where $a(\lambda)=\alpha_0\lambda^2+\alpha_1\lambda+\alpha_2$ and $b(\eta)=\beta_0+\beta_1\eta+\beta_2\eta^2$. Given that $z_{1t}-1\ne 0$ and $z_{2t}\ne 0$ we get
$$\lambda_t(b_0\lambda_t+b_1\eta_t)(c_0+c_1\eta_t)-t^2a(\lambda_t)b(\eta_t)=a_0\lambda_t\eta_t+a_1=0.$$
Let $\lambda_t^{(i)}$ be the roots of this last equations, for $i=1,2,3,4$. Reordering, if necessary, we can admit that
$$\lim_{t\to0}\lambda_t^{(1)}=0,\,\,\,\lim_{t\to0}\lambda_t^{(2)}=\frac{a_1c_1}{a_0c_0},\,\,\,\lim_{t\to0}\lambda_t^{(3)}=\sqrt{\frac{a_1b_1}{a_0b_0}},\,\,\,\lim_{t\to0}\lambda_t^{(4)}=-\sqrt{\frac{a_1b_1}{a_0b_0}}.$$
Solving the system (\ref{sist1}), we get $z_t^{(i)}=(z_{1t}^{(i)},z_{2t}^{(i)},z_{3t}^{(i)})$ where
$$ z_{1t}^{(i)} = \frac{b_0(\lambda_t^{(i)})^2+b_1\eta_t^{(i)}\lambda_t^{(i)}}{b_0(\lambda_t^{(i)})^2+b_1\eta_t^{(i)}\lambda_t^{(i)}+ta(\lambda_t^{(i)})},$$
$z_{1t}^{(i)}=\lambda_t^{(i)}z_{2t}^{(i)}$ and $z_{3t}^{(i)}=1+\eta_t^{(i)}(z_{1t}^{(i)}-1)$.
For $i=1,2$ we obtain
$$\lim_{t\to0}z_{1t}^{(i)} = 1$$
which results that 
$$\lim_{t\to0}z_{t}^{(i)} \in C_2.$$
More precisely, 
$$ \lim_{t\to0}z_{t}^{(1)}=p_4,\quad \lim_{t\to0}z_{t}^{(2)}=(1,1/\lambda_0^{(2)},1),\quad \lambda_0^{(2)}=\frac{a_1c_1}{a_0c_0}.$$
Now, from the second equation of (\ref{sist1}), we get
$$z_{2t}=\frac{tb(\eta_t)}{\lambda_t[(c_0+c_1\eta_t)+tb(\eta_t)]}$$
which results 
$$\lim_{t\to0}z_{2t}^{(i)} = 0$$
for $i=3,4$. Thus, in this situation, 
$$\lim_{t\to0}z_{t}^{(i)} \in C_1.$$

In the hyperplane $H_3$, the singular set of $\fol_t$ contains 7 more points, $p_1\in C_1$ and $p_4\in C_2$ being two of those 7 points.  However, given that the Milnor numbers $\mu(\fol|_{H_3},p_1)=1$ and $\mu(\fol|_{H_3},p_4)=2$ we get
$\mu(\fol,C_1)=5$ and $\mu(\fol,C_2)=6$.

\end{example}

\section{The Milnor number and the Poincar\'e-Hopf index}\label{The Milnor}

Let $\fol$ be  a one-dimensional holomorphic foliation on a neighborhood of $(\mathbb{C}^n,0)$, $n\geq 2$, with an isolated singularity at the origin. 
It is well known  (see \cite{CSL}) that the Milnor number $\mu(\fol,0)$ coincides with the Poincar\'e-Hopf index at $0\in\mathbb{C}^n$ of
  any holomorphic vector field generating $\fol$. In particular, the Poincar\'e-Hopf index of a  holomorphic vector field tangent to $\fol$
 depends only on the foliation $\fol$. This property holds also for non-holomorphic vector fields and will be important to give a brief explanation of this
 fact. Let $v$ and $\tilde{v}$ be two continuous vector fields tangent to $\fol$, both with an isolated singularity at $0\in\mathbb{C}^n$. 
Then, if $B$ is a small ball centered at the origin, there exists a continuous map
 $f\colon \partial B \to \mathbb{C}^*$ such  that $\tilde{v}=f v$ on $\partial B$. But it happens that any such map $f$ is necessarily homotopic to a 
constant map. This implies that the vector fields $v$ and $\tilde{v}$ are homotopic as nowhere zero sections of $T\mathbb{C}^n|_{\partial B}$ and
 therefore their Poincar\'e-Hopf indexes coincide. These arguments also work for defining a Milnor number for isolated singularities of 
continuous orientable 2-dimensional distributions on manifolds, as we can see in \cite{GSV}. 

%Essentially, in \cite{CSL} (see also \cite{GSV})  is proved that   $\mu(\fol,0)$  coincides with the ``Poincar\'e-Hopf index'' at $0\in\mathbb{C}^n$ of any continuous real flow tangent to $\fol$.  This is the key to prove that    $\mu(\fol,0)$  is a topological invariant of the foliation. On the other hand,  and also permit us to define 

\par The Poincar\'e-Hopf index for real vector fields is well understood, even for non-isolated singularities. 
Nevertheless, this fact can not be directly used to define a Milnor number for non-isolated singularities of holomorphic foliations: a foliation is not always defined only by a vector field. Furthermore, even if the foliation 
is defined by a vector field, in general, this vector field will  not be unique, as we have seen in the case of isolated 
singularities above. In that case,  the Milnor number is well defined because any continuous map from $\partial B$ to $\mathbb{C}^*$ is 
homotopically trivial. 
This fact is a particular property of the sphere $\partial B$ and need no to be true if $B$ is a neighborhood of a general connected 
component of the singular set of the foliation.  
\par We recall the definition of the Poincar\'e-Hopf index for non-isolated
 singularities of real vector fields. Let $v$ be a continuous vector field on a manifold $M$ and let $S$ be a compact connected component of the  singular set of $v$. 
Let $T$ be a compact neighborhood of $S$  such there are no singularities on $T\backslash S$ and take a vector field  $\tilde{v}$ on $T$ with
 isolated singularities and such that $\tilde{v}=v$ near $\partial T$. Then the Poincaré-Hopf index of $v$ at $S$ is defined as
$$\textrm{Ind}(v,S)=\sum\limits_{p\in\textrm{Sing}(\tilde{v})}\textrm{Ind}(\tilde{v},p),$$ where $\textrm{Ind}(\tilde{v},p)$ denotes the Poincaré-Hopf 
index of $\tilde{v}$ at the isolated singularity $p$.  Of course, this definition is based on the fact that the sum of indexes of $\tilde{v}$ on $T$
 depends only on the vector field $\tilde{v}|_{\partial T}={v}|_{\partial T}$, that is, only depends on $v$. In fact, that sum depends only  on
 the homotopy class of $v$ in the space of nowhere zero sections of  $TM|_{\partial T}$.  Unfortunately, this property is no longer true for holomorphic foliations, as shown the following examples.
\begin{example}\label{ejemplo1} 
Let $v$ and $\tilde{v}$ be two polynomial vector fields on $\mathbb{C}^2$ with isolated singularities,  both with the same 
linear part 
$$(ax+by)\frac{\partial}{\partial x}+(cx+dy)\frac{\partial}{ \partial y},$$ 
where  $ad-bc\neq 0$.
Let $B$ be a small ball centered at the origin and let $T$ be the complement of $B$ in the projective complex plane $\mathbb{CP}^2$. The vector fields $v$ and $\tilde{v}$ define two holomorphic foliations $\fol$ and $\tilde{\fol}$ with isolated singularities on $T$. Since $v$ and $\tilde{v}$ have the same non-degenerated linear part at $0\in\mathbb{C}^2$, for $B$ small enough the vector fields $v$ and $\tilde{v}$ are homotopic as nowhere zero sections of $T\mathbb{C}^2|_{\partial B}$. This implies that, viewed as continuous distributions, $\fol$ and  $\tilde{\fol}$ are homotopic on $\partial T$. Nevertheless, the sum of Milnor numbers on $T$ for the foliations $\fol$ and $\tilde{\fol}$  are not necessarily the same, because this sum depends on the degree of the corresponding foliation.  In fact, by deforming these foliations we easily obtain an example of two continuous distributions on $T$, coinciding on $\partial T$, but with different sums of Milnor numbers on $T$. This shows that a possible definition of the Milnor number for non-isolated singularities of 2-dimensional distributions does not work in the same way as in the case of vector fields, at least in a general setting.
\end{example}
\begin{remark} Example \ref{ejemplo1} also shows that the Poincar\'e-Hopf Index Theorem for holomorphic vector fields proved in \cite[Theorem 1]{ito} can not be extended to one-dimensional holomorphic foliations.
\end{remark}
\begin{example}\label{ejemplo2} Consider the holomorphic vector field $$v=x^2\frac{\partial}{\partial x}+y^n\frac{\partial}{\partial y}$$ on $\mathbb{C}\times\mathbb{C}$. This vector field extends to a vector field on $\overline{\mathbb{C}}\times\mathbb{C}$ with a unique singularity at $(0,0)$ of Milnor number $2n$. It is easy to see that the  vector  field $v$ is transverse to the boundary of the compact domain $$D=\overline{\mathbb{C}}\times\overline{\mathbb{D}}.$$ Nevertheless, the sum of Poincar\'e-Hopf indexes of $v$ on $D$ --- exactly $2n$ --- is not always equal to $\chi(D)=2$. Thus, this example shows that the surjectivity of the natural morphism $H^1(D,\mathbb{Z})\to H^1(\partial D,\mathbb{Z})$, assumed in \cite{ito} as a hypothesis, cannot be removed.
\end{example}

\section{Topological invariance of the Milnor number}\label{Theorem_1}

It is well known \cite{CSL, GSV} that the Milnor number of an isolated singularity of a holomorphic foliation is a topological invariant. Essentially, this theorem is based upon the following two facts.
\begin{enumerate}
\item A holomorphic foliation on a complex surface $V$ near an isolated singularity $p\in V$ is always generated by a vector field.
\item The isolated singularity $p$ has an arbitrarily small neighborhood $B$ in $V$ such that the set $B^*=B\backslash\{p\}$ has the following property: Any continuous map $f\colon B^*\to\mathbb{C}^*$ is homotopically null. In fact, if $B$ is a ball, then $B$ is homotopically equivalent to the 3-dimensional sphere  and we know that $\pi_3(\mathbb{C}^*)=0$. 

\end{enumerate}
We give a sketch of the proof of the topological invariance of the Milnor number of an isolated singularity. Let $\phi$ be a homeomorphism between a neighborhood of $p$ in $V$ to a neighborhood of $p'$ in $V'$ conjugating the foliations $\fol$ and $\fol'$. Let $Z$ be a holomorphic vector field defining $\fol$ near $p$. For the sake of simplicity, we assume that $\phi$ is a $C^1$ diffeomorphism. Then the vector field $Z'=d\phi.Z$ is a continuous vector field tangent to $\fol'$ with an isolated singularity at $p'$. As we have seen in section \ref{The Milnor}, the Milnor number of $\fol'$ at $p'$ is equal to the Poincar\'e-Hopf index of $Z'$. So the topological invariance of the Milnor number follows from the topological invariance of the Poincar\'e-Hopf index. If $\phi$ is only a homeomorphism we still can define $Z'$ as a local continuous real flow with an isolated singularity at $p'$. In this case, we need to extend the definition of the Poincar\'e-Hopf index for isolated singularities of real continuous flows (see \cite{GSV}), and the proof follows essentially in the same way. The importance of the first fact above is evident because it permits us to reduce  the Milnor number to a Poincar\'e-Hopf index. On the other hand, the second fact above is of capital importance, because it guarantees that the Milnor number of a foliation coincides with the Poincar\'e-Hopf index of any continuous flow tangent to the foliation. 
\par In general, if $\fol$ is a holomorphic foliation on a complex manifold $M$ and $C$ is a connected component  of $\textrm{Sing}(\fol)$, the two facts above are not necessarily true, so the topological invariance of the Milnor number $\mu(\fol,C)$ seems to be a nontrivial problem if $C$ has positive dimension.  We present a partial solution to this problem.

\begin{theorem}\label{c1inv} Let  $\fol$ be a one-dimensional holomorphic foliation on a complex three-dimensional manifold $M$ such that $\emph{Sing}(\fol)$ has codimension bigger than one.   Let $\fol'$ be another holomorphic foliation on a complex  three-dimensional manifold $M'$ topologically equivalent to $\fol$ by an orientation preserving $C^1$ diffeomorphism $\phi\colon M\to M'$. We assume that $\phi$ preserves the natural orientation of the leaves.  Let $C$ be a compact connected component  of $\sing(\fol)$. Suppose that $C$ has arbitrarily small neighborhoods $V$ with $H^1(V,\mathbb{Z})=0$.  Then $$\mu(\fol,C)=\mu(\fol',\phi(C)).$$
\end{theorem}
\par Observe that in the statement of Theorem \ref{c1inv} the equivalence $\phi$ is globally defined on $M$. Nevertheless, the manifold $M$ need not be closed and, in particular,  $M$ could be
a  small neighborhood of $C$. We also note that the existence of arbitrarily small neighborhoods of $C$
with vanishing first cohomology group is obviously fulfilled if $C$ is an isolated singularity. 

It is easy to see that Theorem \ref{c1inv} is a direct consequence of the following two propositions.

\begin{proposition}\label{c1cover}
Let  $\fol$ be a one-dimensional holomorphic foliation on a complex three-dimensional manifold $M$ such that $\emph{Sing}(\fol)$ has codimension bigger than one.   Let $\fol'$ be another holomorphic foliation on a complex three-dimensional manifold $M'$ topologically equivalent to $\fol$ by a $C^1$  diffeomorphism $\phi\colon M\to M'$. We assume that $\phi$ preserves the natural orientation of the leaves.  Let $C$ be a compact connected component  of $\emph{Sing}(\fol)$. Then there exist a neighborhood ${\Omega}$ of $C$ and an isomorphism $g\colon TM|_{\Omega}\to TM'|_{\phi({\Omega})}$ of complex vector bundles with the following properties:
\begin{enumerate}
\item $g$  covers the homeomorphism $\phi|_{\Omega}\colon {\Omega}\to \phi({\Omega})$;
\item if $x\in {\Omega}\backslash \emph{Sing}(\fol)$, then $g(T_x\fol)=T_{\phi (x)}\fol'$.
\end{enumerate} 

\end{proposition}

\begin{proposition}\label{propinv}
Let $M$ be a complex manifold such that $H^1(M,\mathbb{Z})=0$.  Let  $\fol$ be a one-dimensional holomorphic foliation on $M$ such that $\emph{Sing}(\fol)$ has codimension bigger than one.   Let $\fol'$ be another holomorphic foliation on a complex manifold $M'$ topologically equivalent to $\fol$ by an orientation preserving homeomorphism $\phi\colon M\to M'$.  Let $C$ be a compact connected component  of $\emph{Sing}(\fol)$. Suppose that there exists an isomorphism $g\colon TM\to TM'$ of complex vector bundles with the following properties:
\begin{enumerate}
\item $g$  covers the homeomorphism $\phi\colon M\to M'$;
\item if $x\in M\backslash \emph{Sing}(\fol)$, then $g(T_x\fol)=T_{\phi (x)}\fol'$.
\end{enumerate} 
Then
 $$\mu(\fol,C)=\mu(\fol',\phi(C)).$$
\end{proposition}
In order to prove Proposition \ref{c1cover} we need  two lemmas that we state and prove below.
Let $E$ and $E'$ be finite-dimensional complex vector spaces and let 
$\sigma\colon E\to E'$ be a real-linear map. We say that $\sigma$ is complex-antilinear if
 $\sigma(ax)=\bar{a}\sigma(x)$ for all $a\in\mathbb{C}$, $x\in E$. It is well known that any
 real-linear map $\sigma\colon E\to E'$  can be expressed in a unique way as
$$\sigma=\partial\sigma+\bar{\partial}\sigma,$$ were $\partial\sigma\colon E\to E'$ is complex-linear and $\bar{\partial}\sigma\colon E\to E'$ is complex-antilinear. For each 
$t\in[0,1]$ define the real linear map 
$$H_{\sigma}^t=\partial\sigma+t\bar{\partial}\sigma.$$ The family $\{H_{\sigma}^t\}_{t\in[0,1]}$ will be called the \emph{canonical deformation} of $\sigma$. Since $H_{\sigma}^t$ depends continuously on $\sigma$, this canonical deformation will be useful in the construction of deformations of  real isomorphism of complex bundles. Let  $e_1,\ldots,e_n$ be a base of $E$. As a real vector space, $E$ can be endowed with the natural orientation defined by the basis $e_1, ie_1,\ldots, e_n,ie_n$. We do the same with $E'$. A subspace $L\subset E$ is called a complex line if $\dim_{\mathbb{C}}L=1$; in this case, $L=\mathbb{C} v$ for any nonzero element $v\in L$ and, as a real vector space,  the complex line $L$  has the natural orientation defined by the basis $\{v,iv\}$. 

\begin{lemma}\label{constancia} let $\sigma\colon E\to E'$ be any real-linear map between the complex vector spaces $E$ and $E'$. Consider the the canonical deformation $H_{\sigma}^t$  of $\sigma$. Let $L$ and $L'$ be  complex lines in $E$ and $E'$ respectively, and suppose that $\sigma (L)\subset L'$. Then, for each $t\in [0,1]$ we have  $H_{\sigma}^t(L)\subset L'$.  
\end{lemma}
\begin{proof}Let $v\in L$. Since $L\subset L'$, we have that $$\sigma(v)=\partial\sigma(v)+\bar{\partial}\sigma(v)$$ and 
$$\sigma(iv)=i\partial\sigma(v)-i\bar{\partial}\sigma(v)$$ are contained in $L'$. From this we obtain  that $\partial\sigma(v)$ and  $\bar{\partial}\sigma(v)$ are contained in $L'$ and therefore, given $t\in[0,1]$, we have that $$H_{\sigma}^t(v)=\partial\sigma(v)+t\bar{\partial}\sigma(v)$$ is contained in $L'$. 
\end{proof}

\begin{lemma}\label{deformation} Assume that $\dim_{\mathbb{C}}E=\dim_{\mathbb{C}}E'=n\in\mathbb{N}$, let $\sigma\colon E\to E'$ be an orientation preserving real-linear isomorphism, and consider the  canonical deformation $H_{\sigma}^t$  of $\sigma$.  Suppose that there exist $n-1$ linearly independent complex lines $L_1,\ldots L_{n-1}$ in $E$ such that each $L_j$ is mapped by $\sigma$ onto a complex line in $E'$ preserving the natural orientations of complex lines.  Then, for each $t\in [0,1]$, the map $H_{\sigma}^t\colon E\to E'$ is an orientation preserving real-linear isomorphism. In particular, $H_{\sigma}^0=\partial\sigma$ is a complex-linear isomorphism.  
\end{lemma}
\begin{proof}Without loss of generality we can assume that
$ E=E'=\mathbb{C}^n$ and that, for $j=1,\ldots,n-1$,  both $L_j$ and $\sigma(L_j)$ are equal to  the $jth$ complex axis of $\mathbb{C}^n$. Then $\sigma$ preserves each of the first $n-1$ axes and we can express $\sigma$ as a $n\times n$ matrix 
$$\sigma=\begin{bmatrix}A_1&0 & 0&\ldots &0&B_1\\
0&A_2&0&\ldots&0&B_2\\
0&0&A_3&\ldots&0&B_3\\
\vdots&\vdots& \vdots &\ddots&\vdots&\vdots\\
0&0&0&\ldots & A_{n-1}&B_{n-1}\\
0&0&0&\ldots & 0&A_{n}\\
\end{bmatrix},$$ whose entries are real $2\times 2$ matrices. So, it is easy to see that 
$$H_{\sigma}^t=\begin{bmatrix}\partial A_1+t\bar{\partial}A_1&0 & \ldots &0&\partial B_1+t\bar{\partial}B_1\\
0&\partial A_2+t\bar{\partial}A_2&\ldots&0&\partial B_2+t\bar{\partial}B_2\\

\vdots&\vdots&\ddots&\vdots&\vdots\\
0&0&\ldots & \partial A_{n-1}+t\bar{\partial}A_{n-1}&\partial B_{n-1}+t\bar{\partial}B_{n-1}\\
0&0&\ldots & 0&\partial A_{n}+t\bar{\partial}A_{n}\\
\end{bmatrix},$$ hence
$$\det H_{\sigma}^t=\det(\partial A_1+t\bar{\partial}A_1)\ldots \det(\partial A_n+t\bar{\partial}A_n).$$ Therefore, it suffices to show that $$\det(\partial A_j+t\bar{\partial}A_j)>0$$ for all $t\in[0,1]$, $j=1,\ldots,n$.
Since $\sigma$ preserves the orientation of each of the first $n-1$ axes, we have that $\det A_1,\ldots,\det A_{n-1}>0$. Then, since by hypothesis $\det \sigma$ is positive, we also have $\det A_n>0$. Given $j=1,\ldots,n$, there are constants $a,b\in\mathbb{C}$ such that $\partial A_j(z)=az$, $\bar{\partial}A_j(z)=b\bar{z}$,    so by a direct computation we obtain  $$\det(\partial A_j+t\bar{\partial}A_j)=|a|^2-t^2|b|^2\ge |a|^2-|b|^2=\det(A_j)>0.$$

\end{proof}

\subsection{Proof of Proposition \ref{c1cover}.} We start the proof with the following assertion.\\
\noindent\emph{Assertion.} If $p\in C$, then there exist infinitely 
many complex lines $L$ in  $T_p M$ such that $d\phi(p)(L)$ is a complex line in $T_{\phi (p)} M'$. In fact, consider the set $A$ of the complex lines 
$L$ in $T_p M$ such that $L=\lim T_{p_n}\fol$ for some sequence $(p_n)$ in $M\backslash \operatorname{Sing}(\fol)$ with
 $p_n\rightarrow p$. Given any such complex line $L$ in $A$, since $\phi$ maps leaves of $\fol$ to leaves of $\fol'$ we have $$d\phi (p_n)\big( T_{p_n}\fol\big)=T_{\phi(p_n)}\fol'\textrm{ for all }n\in\mathbb{N}.$$ 
 Thus, since $\phi\in C^1$ and the space of complex lines in $TM'$ is closed, the real linear space
 $$d\phi (p)(L)=\lim d\phi ({p_n})\big( T_{p_n}\fol\big)=\lim T_{\phi(p_n)}\fol' $$ is in fact a complex line in $T_{\phi(p)}M'$. So it suffices to prove that the set $A$ is infinite. It is easy to see that the set $A$ is non-empty and connected; thus, if $A$ is finite, we need to assume that  it is unitary, say $A=\{L\}$. From this we obtain the following implication: 
 \begin{equation}\label{tile1}\zeta\rightarrow p,\; \zeta\in M\backslash\operatorname{Sing}(\fol) \implies\lim T_\zeta \fol=L .\end{equation}
Consider holomorphic coordinates $(x,y,z)$ at $p$ such that $p=(0,0,0)$, $L=[0:0:1]$ and take a holomorphic vector field $$a\frac{\partial}{\partial x}+b\frac{\partial}{\partial y}+c\frac{\partial}{\partial z}$$ generating $\fol$ at $p$. The property \eqref{tile1} implies that for $\zeta\in M\backslash\operatorname{Sing}(\fol) $ close enough to $p$ we have $c(\zeta)\neq 0$. Thus, since $\operatorname{Sing}(\fol)$ has codimension $\ge 2$ we deduce that $c(\zeta)\neq 0$ for all $\zeta$ in a neighborhood of $p$, hence 
$\fol$ is regular at $p$, which is a contradiction and the assertion is proved.
\par  By the assertion above we can take two different complex lines $L_1$ and $L_2$ in $T_p M$ such
 that $d\phi(p)(L_1)$ and $d\phi(p)(L_2)$ are complex lines in $T_{\phi (p)} M'$. Observe that $L_1\neq L_2$ implies that $L_1, L_2$ are linearly independent. 
 So, it follows by Lemma \ref{deformation} that, for any $p\in S$, the canonical deformation
  $H_{d\phi (p)}^t$ of $d\phi (p)$ is an orientation preserving isomorphism for each $t\in[0,1]$. Since $S$ is compact and $d\phi$ is
   continuous, we can find a neighborhood $\Omega$ of $S$ in $M$ such that the
    canonical deformation $H_{d\phi (\zeta)}^t$ of $d\phi (x)$ is an orientation preserving isomorphism for each $t\in[0,1]$  $x\in\Omega$. This allows us to deform isotopically the real isomorphism of 
    complex bundles $$d\phi: TM|_{\Omega}\to TM'|_{\phi (\Omega)}$$ 
    into the complex isomorphism
$$g\colon TM|_{\Omega}\to TM'|_{\phi (\Omega)}$$ defined by 
$g|_{T_x M}= \partial(d\phi (x))$ for all $x\in \Omega$. The first statement of Proposition \ref{propinv} is clearly satisfied by $g$. Let $x\in \Omega\backslash \operatorname{Sing}(\fol)$. Since $d\phi (x) (T_x\fol)=T_{\phi (x)}\fol'$, it follows from Lemma \ref{constancia} that $g (T_x\fol)=T_{\phi (x)}\fol'$, so the second statement is proved.\qed
\begin{remark}\label{obs1}
Observe that three pairwise distinct complex lines can be linearly dependent. For this reason, our proof of Proposition \ref{c1cover} only works in dimension three. 
\end{remark}

\subsection{Proof of Proposition  \ref{propinv}.} 
By \cite[Theorem 5.1]{GSV}, the isomorphism $$\phi^*\colon H^2(M,\mathbb{Z})\to H^2(M',\mathbb{Z})$$ induced by $\phi$ maps the Chern class of $T_{\fol}$ onto the Chern class of $T_{\fol'}$. Since the Chern Class is a complete invariant  in the classification of complex line bundles up to isomorphism, we have that there exists an isomorphism $\xi\colon T_{\fol}\to T_{\fol'}$ covering the map $\phi\colon M\to M'$. Consider the dual of the inverse isomorphism $\xi^{-1}$, 
$$f=(\xi^{-1})^*:T_{\fol}^*\to T_{\fol'}^*$$ and  $$h:=g\otimes f\colon TM\otimes T_{\fol}^*\to  TM'\otimes T_{\fol'}^*.$$  Let $s$ be a section of $TM\otimes T_{\fol}^*$ defining $\fol$. Denote by $s_0$ and $s_0'$ the zero sections of $TM\otimes T_{\fol}^*$ and $ TM'\otimes T_{\fol'}^*$, respectively.  By the topological invariance of the intersection number we have that 
\begin{equation}\label{milnor_1}
\mu(\fol,C)=i_{C}(s,s_0)=i_{C'}(h\circ s,s_0'),
\end{equation} 
where $C'=\phi(C)$.\\
\noindent\emph{Assertion.} Let $s'$ be a section of $TM'\otimes T_{\fol'}^*$ defining $\fol'$. Then there exists $\theta\colon M\backslash C \to\mathbb{C}^*$ continuous such that $$h\big(s(x)\big)=\theta(x) s'\big(\phi(x)\big),$$ for all $x\in\ M\backslash C$. Fix $x\in M\backslash C$. Since $\fol$ is defined by the section $s$, there exists
 $\zeta\in (T_{\fol})_x$ and $v\in T_x\fol\subset T_xM$ such that $s(x)=v\otimes \zeta$, so $$h\big(s(x)\big)=g(v)\otimes f(\zeta).$$
Since $\fol'$ is defined by the section $s'$, there exist $\zeta'\in (T_{\fol'})_{\phi(x)}$ and $v'\in T_{\phi(x)}\fol'\subset T_{\phi(x)}M'$ such that $s'\big(\phi(x)\big)=v'\otimes \zeta'$. Since $v\in T_x\fol$, by hypothesis we have that $g(v)\in T_{\phi(x)}\fol'$, so there exists $\alpha\in\mathbb{C}^*$ such that  $g(v)=\alpha v'$. Therefore $$h\big(s(x)\big)= g(v)\otimes f(\zeta)=(\alpha v')\otimes f(\zeta)=\alpha  (v'\otimes f(\zeta))=\theta (v'\otimes\zeta')=\theta s'\big(\phi(x)\big)$$ for some $\theta\in\mathbb{C}^*$ (here $\alpha f(\zeta)=\theta \zeta'$). It is easy to see that $\theta$ depends continuously on $x\in M\backslash C$.
\par Since $C$ has complex codimension at least 2 in $M$, we have  $H^1(M\backslash C,\mathbb{Z})=H^1(M,\mathbb{Z})=0$. From this fact, it is easy to prove that the map $\theta:M\backslash C\to\mathbb{C}^*$ is homotopic to a constant map. Then, far from $C'$,   the section $h(s)$ of $TM'\otimes T_{\fol'}^*$ can be deformed to coincide with $s'$ with no variation in the intersection number with the zero section. Thus, we deduce that 
$$i_{C'}(h\circ s,s_0')=i_{C'}(s',s_0')=\mu(\fol',C')$$ and therefore $\mu(\fol',C')=\mu(\fol,C),$ by equation (\ref{milnor_1}).\qed 

\vspace{1cm}

\noindent {\bf Acknowledgments.} 
The authors wish to express his gratitude to Miguel Rodr\'iguez Pe\~na for several helpful comments concerning to work.

\end{document}